\documentclass[]{amsart}
\usepackage{latexsym,xy}
\xyoption{all}
\usepackage{amscd,amssymb,latexsym,subfigure,verbatim,hyperref,epsfig,supertabular}
\usepackage{tikz}
\usetikzlibrary{matrix,arrows,decorations.pathmorphing}

\newtheorem{defn0}{Definition}[section]
\newtheorem{prop0}[defn0]{Proposition}
\newtheorem{conj0}[defn0]{Conjecture}
\newtheorem{thm0}[defn0]{Theorem}
\newtheorem{lem0}[defn0]{Lemma}
\newtheorem{corollary0}[defn0]{Corollary}
\newtheorem{example0}[defn0]{Example}
\newtheorem{remark0}[defn0]{Remark}
\newtheorem{que0}[defn0]{Question}

\newenvironment{defn}{\begin{defn0}}{\end{defn0}}
\newenvironment{prop}{\begin{prop0}}{\end{prop0}}
\newenvironment{conj}{\begin{conj0}}{\end{conj0}}
\newenvironment{thm}{\begin{thm0}}{\end{thm0}}
\newenvironment{lem}{\begin{lem0}}{\end{lem0}}
\newenvironment{que}{\begin{que0}}{\end{que0}}
\newenvironment{cor}{\begin{corollary0}}{\end{corollary0}}
\newenvironment{exm}{\begin{example0}\rm}{\end{example0}}
\newenvironment{rmk}{\begin{remark0}\rm}{\end{remark0}}

\newcommand{\V}{{\bf V }}
\renewcommand{\P}{{\mathbb{P}}}
\newcommand{\Z}{{\mathbb{Z}}}
\newcommand{\A}{{\mathbb{A}}}
\newcommand{\FF}{{\mathbb{F}}}
\newcommand{\Oc}{{\mathcal{O}}}
\newcommand{\OPP}{{\mathcal{O}_{\mathbb{P}^1 \times \mathbb{P}^1}}}
\newcommand{\PP}{{\mathbb{P}^1 \times \mathbb{P}^1}}
\newcommand{\Pone}{{\mathbb{P}^1}}
\newcommand{\OPone}{{\mathcal{O}_{\mathbb{P}^1}}}
\newcommand{\spn}{{\mathrm{Span}}}
\newcommand{\syz}{{\mathrm{Syz}}}
\newcommand{\DET}{{\mathrm{det}}}
\newcommand{\rank}{{\mathrm{rank}}}
\newcommand{\im}{{\mathrm{im}}}
\newcommand{\depth}{{\mathrm{depth}}}
\newcommand{\coker}{{\mathrm{coker}}}
\newcommand{\Ass}{{\mathrm{Ass}}}
\newcommand{\Sing}{{\mathrm{Sing}}}
\newcommand{\Pic}{{\mathrm{Pic}}}
\newcommand{\HP}{{\mathrm{HP}}}
\newcommand{\mm}{{\mathfrak{m}}}

\numberwithin{equation}{section}

\begin{document}
\title{A special case of Postnikov-Shapiro Conjecture}

\author{Jimmy Jianyun Shan}
\address{Department of Mathematics, University of Illinois, 
Urbana, IL 61801}
\email{shan15@math.uiuc.edu}

\begin{abstract}
For a graph $G$, Postnikov-Shapiro \cite{PS04} construct two ideals $I_G$ and $J_G.$
$I_G$ is a monomial ideal and $J_G$ is generated by powers of linear forms. 
They proved the equality of their Hilbert series 
and conjectured that the graded Betti numbers are equal.
When $G=K_{n+1}^{l,k}$ is the complete graph on the vertices $\{0,1,\cdots, n\}$ with the edges $e_{i, j},$ $i, j\neq 0,$ of multiplicity $k$
and the edges $e_{0, i}$ of multiplicity $l,$ for two non-negative integers $k$ and $l,$ 
they gave an explicit formula for the graded Betti numbers of $I_G,$ which are conjecturally the same for $J_G.$
We prove this conjecture in the case $n=3,$ which was also conjectured by Schenck \cite{S04}.
\end{abstract}

\maketitle

\section{Introduction}
Let $R = \mathbb{K}[x_1, \cdots, x_n]$ be a polynomial ring over a field $\mathbb{K}$ of characteristic 0. We consider the following families
of $2^n-1$ generated ideals:
\begin{gather*}
 I_{\phi} = \langle x_1^{\phi(1)},\cdots, x_n^{\phi(1)},(x_1 x_2)^{\phi(2)},\cdots,(x_{i_1} \cdots x_{i_r})^{\phi(r)},\cdots \rangle,\\
 J_{\phi} = \langle x_1^{\phi(1)},\cdots, x_n^{\phi(1)},(x_1+x_2)^{2\phi(2)},\cdots,(x_{i_1}+ \cdots+ x_{i_r})^{r\phi(r)},\cdots \rangle,
\end{gather*}
where $\phi$ is a linear degree function:
$$ \phi(r) = l+k(n-r) > 0, l,k \in \mathbb{N}.$$

These ideals are special cases of the ideals $I_G,$ $J_G$ constructed from a graph $G$ by Postnikov-Shapiro \cite{PS04}. When $G=K_{n+1}^{l,k},$ 
the complete graph on the vertices $\{0,1,\cdots, n\}$ with the edges $e_{i, j},$ $i, j\neq 0,$ of multiplicity $k$
and the edges $e_{0, i}$ of multiplicity $l,$ for two non-negative integers $l$ and $k,$ $I_G =I_{\phi} $ and $J_G = J_{\phi}.$ Postnikov-Shapiro \cite{PS04} proved that 
the ideals $I_{\phi}$ and $J_{\phi}$ have the same Hilbert series. They also gave the following minimal free resolution of $R/I_{\phi}$ 
\begin{equation} \label{cpxI}
... \longrightarrow C_3 \longrightarrow C_2 \longrightarrow C_1 \longrightarrow C_0 = R \longrightarrow R/I_{\phi} \longrightarrow 0, 
\end{equation}
with 
\begin{equation} \label{eq1a}
 C_i = \bigoplus_{l_1,l_2,\cdots,l_i} R(-d(l_1,\cdots,l_i))^{\binom{n}{l_1,\cdots,l_i}},
\end{equation}
where the direct sum is over $l_1,\cdots,l_i \geq 1$ such that $l_1 + \cdots + l_i \leq n,$
\[
 d(l_1,\cdots, l_i) =l_1 \phi(l_1) + l_2 \phi(l_1+l_2) + \cdots + l_i \phi(l_1+\cdots+l_i),  
\]
and 
\[
 \binom{n}{l_1,\cdots,l_i} = \frac{n!}{l_1! \cdots l_i! (n-l_1-\cdots-l_i)!}
\]
is the multinomial coefficient. This means the graded Betti numbers of $I_{\phi}$ are given by 
\[
 b_{i,d(l_1,\cdots, l_i)} = \binom{n}{l_1,\cdots,l_i}.
\]
Moreover, the $i$-th Betti number is given by 
\[
b_{i}(I_{\phi}) = i! S(n+1,i+1), 
\]
where $S(n+1,i+1)$ is the Stirling number of the second kind, i.e., the number of partitions of 
the set $\{0,1,\cdots, n\}$ into $i+1$ nonempty subsets.  More intrinsically, a minimal free resolution of the ideal
$I_{\phi}$ is given by the cellular complex corresponding to the simplicial complex $\Delta,$ which is the barycentric
subdivision of the $(n-1)$-dimensional simplex.

\begin{conj}\cite{PS04} \label{conj1}
The graded Betti numbers of $J_{\phi}$ are also given by \eqref{eq1a}.
\end{conj}

In the case $n=3,$ the ideals $I_{\phi}$ and $J_{\phi}$ are given by
\begin{align*} \label{eqIJ3}
I_{\phi}& = \langle x^{l+2k}, y^{l+2k}, z^{l+2k}, (xy)^{l+k}, (xz)^{l+k}, (yz)^{l+k}, (xyz)^{l} \rangle, \\
J_{\phi}& = \langle x^{l+2k}, y^{l+2k}, z^{l+2k}, (x+y)^{2l+2k}, (x+z)^{2l+2k}, (y+z)^{2l+2k}, (x+y+z)^{3l} \rangle. 
\end{align*}

Schenck \cite{S04} computed the Hilbert series of $R/J_{\phi},$ using ideals of fatpoints: 
\begin{equation} \label{hs}
HS(R/J_{\phi},t)=\frac{1-3t^{l+2k} - 3t^{2l+2k} - t^{3l} + 6 t^{2l+3k} + 6t^{3l+2k} - 6t^{3l+3k}}{(1-t)^3}.
\end{equation}

and conjectured that:
\begin{conj} \label{conj2}
In the case $n=3,$ the minimal free resolution of $J_\phi$ is:
$$0 \xrightarrow{} R(-3l-3k)^6 \xrightarrow{}
\begin{array}{c}
 R(-3l-2k)^6\\
\oplus\\
R(-2l-3k)^6 \\
\end{array}
\xrightarrow{}
\begin{array}{c}
R(-l-2k)^3\\ 
\oplus\\
R(-2l-2k)^3 \\
 \oplus \\
R(-3l)\\
\end{array}
\xrightarrow{}
R 
\xrightarrow{}
R/J_\phi \xrightarrow{} 0.$$
In other words, the graded Betti numbers of $J_{\phi}$ are the same as those of $I_{\phi}.$
\end{conj}

\begin{rmk}
We restrict to the case $char(\mathbb{K}) = 0,$ because in low characteristic, the ideals $I_{\phi}$ and $J_{\phi}$ don't 
even have the same Hilbert series, as pointed out by Schenck \cite{S04}. Also because of the use of inverse systems in our proof, 
which is more complicated in positive characteristic.
\end{rmk}

In this paper, we prove Conjecture \ref{conj2}. Because of the form of the Hilbert series, we only need to show all the first syzygies 
or the second syzygies of $J_{\phi}$ are predicted by Conjecture \ref{conj2}. We show that 
the first syzygies of the ideal $J_{\phi}$ are given by 
\[
  R(-2l-3k)^6 \oplus R(-3l-2k)^6,
\]
which means that there are six syzygies of degree $2l+3k$ and six syzygies of degree $3l+2k.$ 
There are two main ideas in this approach; the first is to construct explicitly 
the first syzygies of the ideal $J_{\phi}$ from its subideals, and the second is to show the constructed syzygies
generate all the first syzygies. The main difficulty is to show that there are no first syzygies of degree bigger than 
$\max(2l+3k,3l+2k).$ We prove this by showing that there are no second syzygies of the same degree.
Interestingly, we use fatpoints in the proof. 

We stress that most work (see \cite{EI95}, \cite{G96}, and \cite{I97} for example) on ideals generated by powers of linear forms uses Macaulay Inverse Systems,
to translate into questions about fatpoints. The proof of the above conjecture can be seen as a step to understanding ideals generated by powers of 
linear forms from a new perspective. 

\smallskip
\pagebreak

\section{The Proof of Conjecture \ref{conj2}} \label{sec2}


\subsection{Construction of syzygies} \label{cs}

Consider the following six subideals $J_i$ of $J_{\phi},$ $i =0, 1, 2, 3, 4, 5,$
\begin{align*}
J_0 &= \langle x^{l+2k}, y^{l+2k}, (x+y)^{2l+2k} \rangle,    \\ 
J_1 &= \langle x^{l+2k}, z^{l+2k}, (x+z)^{2l+2k} \rangle,    \\ 
J_2 &= \langle y^{l+2k}, z^{l+2k}, (y+z)^{2l+2k} \rangle,    \\  
J_3 &= \langle x^{l+2k}, (y+z)^{2l+2k}, (x+y+z)^{3l} \rangle,\\  
J_4 &= \langle y^{l+2k}, (x+z)^{2l+2k}, (x+y+z)^{3l} \rangle,\\  
J_5 &= \langle z^{l+2k}, (x+y)^{2l+2k}, (x+y+z)^{3l} \rangle.    
\end{align*}

In the polynomial ring $R,$ the ideal 
\[
 J_0 = \langle x^{l+2k}, y^{l+2k}, (x+y)^{2l+2k} \rangle
\]
is codimension two and has the following minimal free resolution by the Hilbert-Burch Theorem \cite{EI95}, 
\begin{small}
\[
0 \longrightarrow R(-2l-3k)^2 \xrightarrow{ \phi_{0} } 
\begin{array}{c}
 R(-l-2k)^2\\
\oplus\\
R(-2l-2k) \\
\end{array}\!
\xrightarrow{\psi_{0}}
R
\xrightarrow{} R/{J_0}
\xrightarrow{} 0,
\]
\end{small}
where 
\[
\psi_{0} = \left[ \!
\begin{array}{ccc}
x^{l+2k} & y^{l+2k} & (x+y)^{2l+2k}   
\end{array}\! \right].
\]
and 
\[
\phi_{0} = \left[ \!
\begin{array}{cc}
a_1 & a_2  \\
b_1 & b_2  \\
c_1 & c_2  \\
\end{array}\! \right].
\]
Here $a_1,a_2,b_1,b_2$ are polynomials of degree $l+k,$ $c_1,c_2$ are polynomials of degree $k.$  
The entries of $\phi_0$ satisfy the following equations:
\begin{equation}\label{eqJ0}\\
\begin{aligned}
b_1 c_2 - b_2 c_1 &= \alpha_0  x^{l+2k},   \\
a_1 c_2 - a_2 c_1 &=  -\alpha_0 y^{l+2k},  \\
a_1 b_2 - a_2 b_1 & = \alpha_0 (x+y)^{2l+2k},
\end{aligned} 
\end{equation}
for some nonzero constant $\alpha_0.$ Moreover, we have
\begin{equation} \label{col1}
 \langle a_1, a_2 \rangle  = \langle y^{l+2k},(x+y)^{2l+2k} \rangle: x^{l+2k}.
\end{equation}

The two syzygies of the ideal $J_0,$ given by $(a_1,b_1,c_1)^{t}$ and $(a_2,b_2,c_2)^{t},$ are of degree $2l+3k.$ 
They can be naturally extended to syzygies of the ideal $J_{\phi} $ as follows, 
\begin{align*} \label{sz12}
 s_1 & = (a_1,b_1,0,c_1, 0,0,0)^t, \\
 s_2 & = (a_2,b_2,0,c_2, 0,0,0)^t.
\end{align*}
Therefore we have obtained two first syzygies of degree $2l+3k$ of the ideal $J_{\phi}.$ 

The other two ideals $J_1, J_2$ have completely similar minimal free resolutions, with the matrix of the first differential given by
\[ \phi_{1} = 
\left[ \!
\begin{array}{cc}
d_1 & d_2  \\
e_1 & e_2  \\
f_1 & f_2  \\
\end{array}\! \right], \mbox{ }
\phi_{2} = 
\left[ \!
\begin{array}{cc}
g_1 & g_2  \\
h_1 & h_2  \\
k_1 & k_2  \\
\end{array}\! \right].
\]

The entries of $\phi_1, \phi_2$ satisfy the following equations:
\begin{equation}\label{eqJ12}
\begin{aligned}
e_1 f_2 - e_2 f_1 &= \alpha_1  x^{l+2k},   &  h_1 k_2 - h_2 k_1 &= \alpha_2  y^{l+2k},              \\ 
d_1 f_2 - d_2 f_1 &= -\alpha_1 z^{l+2k},   &  g_1 k_2 - g_2 k_1 &= -\alpha_2  z^{l+2k},             \\
d_1 e_2 - d_2 e_1 &= \alpha_1 (x+z)^{2l+2k},&  g_1 h_2 - g_2 h_1 &= \alpha_2  (y+z)^{2l+2k},       
\end{aligned} 
\end{equation}
for some nonzero constant $\alpha_1, \alpha_2. $ 

We also have 
\begin{equation} \label{col2}
 \langle d_1, d_2 \rangle  = \langle z^{l+2k},(x+z)^{2l+2k} \rangle: x^{l+2k}.
\end{equation}

The two syzygies of $J_1$ and those of $J_2$ can also be extended to syzygies of the ideal $J_{\phi},$ given by 
\begin{align*} \label{sz12}
 s_3 & = (d_1,0,e_1,0,f_1,0,0)^t, \\
 s_4 & = (d_2,0,e_2,0,f_2,0,0)^t, \\
 s_5 & = (0,g_1,h_1,0,0,k_1,0)^t, \\
 s_6 & = (0,g_2,h_2,0,0,k_2,0)^t.
\end{align*}
Therefore, we have constructed six first syzygies of degree $2l+3k$ from the ideals $J_0, J_1, J_2.$ In \S\ref{gm}, we 
show these syzygies are independent. 

To construct six first syzygies of degree $3l+2k,$ we consider the subideals $J_3,$ $J_4,$ and $J_5.$ For example, the ideal
\[
 J_3  = \langle x^{l+2k}, (y+z)^{2l+2k}, (x+y+z)^{3l} \rangle
\]
is essentially an ideal in two variables $x, y+z$ and has a Hilbert-Burch resolution,
\begin{small}
\[
0 \longrightarrow R(-3l-2k)^2 \xrightarrow{ \phi_{3} } 
\begin{array}{c}
R(-l-2k)\\
\oplus\\
R(-2l-2k) \\
\oplus\\
R(-3l) 
\end{array}\!
\xrightarrow{\psi_{3}}
R
\xrightarrow{} R/{J_3}
\xrightarrow{} 0,
\]
\end{small}
where the matrix of differential is given by
\[
\psi_{3} = \left[ \!
\begin{array}{ccc}
x^{l+2k} & (y+z)^{2l+2k} & (x+y+z)^{3l}  
\end{array}\! \right],
\]
and 
\[
\phi_{3} = \left[ \!
\begin{array}{cc}
A_1 & A_2  \\
B_1 & B_2  \\
C_1 & C_2  \\
\end{array}\! \right].
\]
Here $A_1,A_2$ are polynomials in $x,y+z$ of degree $2l,$ $B_1,B_2$ are of degree $l,$  
$C_1,C_2$ are of degree $2k.$ Similarly, the entries of $\phi_3$ satisfy the equations:
\begin{equation}\label{eqJ3}
\begin{aligned} 
B_1 C_2 - B_2 C_1 & =  \beta_0 x^{l+2k}, \\ 
A_1 C_2 - A_2 C_1 & =  -\beta_0 (y+z)^{2l+2k}, \\
A_1 B_2 - A_2 B_1 & =  \beta_0 (x+y+z)^{3l},
\end{aligned} 
\end{equation}
for some nonzero constant $\beta_0.$ Moreover, we have 
\begin{equation} \label{col3}
 \langle A_1, A_2 \rangle  = \langle (y+z)^{2l+2k},(x+y+z)^{3l} \rangle: x^{l+2k}.
\end{equation}

The two syzygies of the ideal $J_3,$ given by $(A_1, B_1, C_1)^t$ and $(A_2, B_2, C_2)^t$ are of
degree $3l + 2k.$ They can also be extended to syzygies of the ideal $J_{\phi}$ as
follows
\begin{align*} \label{sz34}
 s_7 & = (A_1,0,0,0,0,B_1,C_1)^t, \\
 s_8 & = (A_2,0,0,0,0,B_2,C_2)^t.
\end{align*}

The ideals $J_4, J_5$ have completely similar minimal free resolutions with their matrices of first differentials given by
\[ \phi_4 = 
\left[ \!
\begin{array}{cc}
D_1 & D_2  \\
E_1 & E_2  \\
F_1 & F_2  \\
\end{array}\! \right], \mbox{ }
\phi_5 =\left[ \!
\begin{array}{cc}
G_1 & G_2  \\
H_1 & H_2  \\
K_1 & K_2  \\
\end{array}\! \right]. 
\]
Here $D_1,D_2,E_1, E_2,F_1,F_2$ are polynomials in $y, x+z$ and $G_1,G_2,H_1,H_2,K_1,K_2$ are polynomials in $z, x+y.$ 
They satisfy equations similar to Equations \eqref{eqJ3}. The two syzygies of $J_4$ and those of $J_5$ are of degree $3l+2k,$ too. 
They can be extended to syzygies of the ideal $J_{\phi},$ given by 
\begin{align*} \label{sz12}
 s_9    & = (0,D_1,0,0,E_1,0,F_1)^t, \\
 s_{10} & = (0,D_2,0,0,E_2,0,F_2)^t, \\
 s_{11} & = (0,0,G_1,H_1,0,0,K_1)^t, \\
 s_{12} & = (0,0,G_2,H_2,0,0,K_2)^t.
\end{align*}
Therefore, we have constructed six first syzygies of degree $3l+2k.$ 

\subsection{Minimal generators of the constructed syzygies} \label{gm}
Now we show the syzygies of degree $2l+3k$ and $3l+2k$ constructed above are minimal generators of the first syzygies of the ideal $J_{\phi};$ 
and there are no other first syzygies of degree at most $\max(2l+3k,3l+2k).$ 
For that purpose, we make use of the structure of the Betti diagram and the Hilbert series of $J_{\phi}.$ 
We devide our analysis into three cases, depending on $l,k.$

Case 1: $l=k.$ This case is trivial, since $2l + 3k = 3l + 2k,$ the constructed syzygies are of the same degree. 

Case 2: $k<l.$ So $2l+3k < 3l+2k.$ In this case, the six first syzygies of degree $2l+3k$ must be minimal and it is impossible
to have first syzygies of degree less than $2l+3k,$ since there are no second syzygies of the same degree to cancel those first syzygies. 

Now we show it is also impossible to have first syzygies of degree $s,$ $s=2l+3k+1, \cdots, 3l+2k-1.$ Starting with $s= 2l+3k+1,$ suppose there are $k_s$ first syzygies of degree $s,$
there must be $k_s$ second syzygies of degree $s$ of $J_{\phi},$ since there is no term $t^s$ in the numerator of the Hilbert series of $R/J_{\phi}.$ Those potential second syzygies
of degree $s$ must be the syzygies of the six syzygies of degree $2l+3k.$  However, there is no such syzygy of degree $s < 3l+2k$ by the following lemma.

\begin{lem} \label{mylm}
The degree of the syzygies of the six syzygies of degree $2l+3k$ is at least $3l+6k.$ 
\end{lem}

\begin{proof}
The six syzygies of degree $2l+3k$ are the columns of the following matrix 
\[A=
\left[ \!
\begin{array}{cccccccccccc}
a_1  &a_2  &d_1  & d_2 & 0   & 0    \\
b_1  &b_2  &0    & 0   & g_1 & g_2  \\
0 & 0      &e_1  & e_2 & h_1 & h_2  \\
c_1 &c_2   &0    & 0   & 0   & 0    \\
0 & 0      &f_1  & f_2 & 0   & 0    \\
0 & 0      &0    & 0   & k_1 & k_2  \\
0 & 0      &0    & 0   & 0   & 0   
\end{array}\! \right]. 
\]

The syzygies of these six syzygies are just the column vectors $v = (v_1, v_2, v_3, v_4, v_5,v_6)^{t},$ where each component $v_i$ is a homogeneous polynomial in $x,y,z,$ such that
\[
 A v = 0.
\]
Writting explicitly, we have
\begin{equation} \label{eqs1}
  \left\{
  \begin{array}{l }
    a_1 v_1 +a_2 v_2 + d_1 v_3 + d_2 v_4 =0 \\
    b_1 v_1 +b_2 v_2 + g_1 v_5 + g_2 v_6 =0 \\
    e_1 v_3 +e_2 v_4 + h_1 v_5 + h_2 v_6 =0 \\
    c_1 v_1 +c_2 v_2  =0 \\
    f_1 v_3 +f_2 v_4  =0 \\
    k_1 v_5 +k_2 v_6  =0 
  \end{array} \right.
\end{equation}
 
Since $c_1,c_2$ are co-prime from Equation \eqref{eqJ0}, the fourth equation implies  that
\[
 (v_1, v_2) = p_1 (-c_2, c_1),
\]
for some polynomial $p_1.$ 

Similarly, we have 
\begin{align*}
 (v_3, v_4)& = p_2 (-f_2, f_1), \\
 (v_5, v_6)& = p_3 (-k_2, k_1),
\end{align*}
for some polynomial $p_2, p_3$ from  the fifth and the last equation, respectively.

Substitute the $v_1,v_2, v_3, v_4, v_5, v_6$ into the first three equations, we get
\begin{equation}
  \left\{
  \begin{array}{l }
    p_1(-a_1c_2 +a_2 c_1) + p_2(-d_1 f_2 + d_2 f_1) =0 \\
    p_1(-b_1c_2 +b_2 c_1) + p_3(-g_1 k_2 + g_2 k_1) =0 \\
    p_2(-e_1f_2 +e_2 f_1) + p_3(-h_1 k_2 + h_2 k_1) =0 
  \end{array} \right.
\end{equation}

By Equations \eqref{eqJ0}, \eqref{eqJ12}, and \eqref{eqJ3}, the above three equations are 
\begin{equation}
  \left\{
  \begin{array}{l }
    p_1(\alpha_0 y^{l+2k})  + p_2(\alpha_1 z^{l+2k}) =0 \\
    p_1(-\alpha_0 x^{l+2k}) + p_3(\alpha_2 z^{l+2k})=0 \\
    p_2(-\alpha_1 x^{l+2k}) + p_3(-\alpha_2 y^{l+2k}) =0 
  \end{array} \right.
\end{equation}

The only solution to these equations is 
\begin{equation} \label{eqp}
 p_1 = c \alpha_1 \alpha_2 z^{l+2k}, p_2 = - c \alpha_0 \alpha_2 y^{l+2k}, \text{ and } p_3 = c \alpha_0\alpha_1 x^{l+2k},
\end{equation}
for some nonzero polynomial $c,$ possibly constant. 

Therefore, the only nonzero syzygies of the six syzygies of degree $2l+3k$ are 
\[
 v = (-c_2 p_1, c_1 p_1, -f_2 p_2, f_1 p_2, -k_2 p_3, k_1 p_3)^t,
\]
with $p_1,p_2,p_3$ given in Equation \eqref{eqp}. Since 
\begin{gather}
 \deg p_1 = \deg p_2 = \deg p_3  \geq l+2k, \\
 \deg c_i = \deg f_i = \deg k_i  = k, \text{ for } i = 1,2.
\end{gather}
Each component of $v$ is of degree at least $l+3k.$ Since the six syzygies are of degree $2l+3k,$ the degree of the syzygies 
of the six syzygies is at least $3l+6k.$
\end{proof}

Therefore there are no first syzygies of degree $s$ where $2l+3k < s < 3l+2k.$ Again, by the Hilbert series, 
the six syzygies of degree $3l+2k$ must be minimal. 

Case 3:$k>l.$ The analysis is similar to the case $k<l.$ In this case, $3l+2k < 2l+3k.$ There are no first syzygies
of degree less than $3l+2k$ and the six first syzygies of degree $3l+2k$ are minimal. 

There are also no first syzygies of degree $s$ such that $3l+2k < s< 2l+3k.$ Suppose not, then there would be second syzygies 
of the syzygies of degree $3l+2k,$ which are the columns of the matrix 
\[B=
\left[ \!
\begin{array}{cccccccccccc}
  & A_1 & A_2 & 0   & 0   & 0   & 0 \\
  & 0   & 0   & D_1 & D_2 & 0   & 0 \\
  & 0   & 0   & 0   & 0   & G_1 & G_2 \\
  & 0   & 0   & 0   & 0   & H_1 & H_2 \\
  & 0   & 0   & E_1 & E_2 & 0   & 0 \\
  & B_1 & B_2 & 0   & 0   & 0   & 0 \\
  & C_1 & C_2 & F_1 & F_2 & K_1 & K_2 
\end{array}\! \right]. 
\]
The syzygies of the six syzygies are the vectors $w = (w_1,w_2,w_3,w_4,w_5,w_6)^t$ such that $B w= 0.$ 

\begin{lem}
The only solution $w$ to the equation $B w= 0$ is $w = 0.$  
\end{lem}

\begin{proof}
Writting the equation $B w= 0$ explicitly, we have
\begin{equation} \label{eqs2}
  \left\{
  \begin{array}{ c }
    A_1 w_1 +A_2 w_2 =0 \\
    D_1 w_3 +D_2 w_4 =0 \\
    G_1 w_5 +G_2 w_6 =0 \\
    H_1 w_5 +H_2 w_6 =0 \\
    E_1 w_3 +E_2 w_4 =0 \\
    B_1 w_1 +B_2 w_2 =0 \\
    C_1 w_1 +C_2 w_2 + F_1 w_3 +F_2 w_4 + K_1 w_5 +K_2 w_6 =0 
  \end{array} \right.
\end{equation}
The first and the sixth equation together imply that $w_1 = w_2 = 0,$ since 
\[
\det \left[ \!
\begin{array}{cc}
   A_1 & A_2  \\
   B_1 & B_2  \\
\end{array}\! \right] = \beta_0(x + y + z)^{3l},
      \]
by Equation \eqref{eqJ3}. Similarly, the second and the fifth equation imply that $w_3 = w_4 = 0,$ since 
\[
\det \left[ \!
\begin{array}{cc}
   D_1 & D_2  \\
   E_1 & E_2  \\
\end{array}\! \right] = \beta_1(x + y + z)^{3l}.
\]
We also have $w_5 = w_6 = 0$ from the third and the fourth equation.
\end{proof}

This shows that the constructed syzygies of degree $2l+3k$ and $3l+2k$ are minimal. 

\subsection{No higher degree first syzygies} \label{comp}
In this subsection, we show the syzygies constructed generate all the first syzygies of the ideal $J_{\phi},$ by proving that 
there are no generators of first syzygies of degree bigger than $\max(2l+3k, 3l+2k).$ 
The argument is similar to show that there are no other first syzygies of degrees at most $\max(2l+3k, 3l+2k);$ in other words,  
we show that there are no second syzygies of $J_{\phi}$ of degree bigger than $\max(2l+3k, 3l+2k),$ except those of degree $3l+3k.$ 

Because the ideal $J_{\phi}$ is Artinian, its regularity is equal to the maximum degree $d$ such that $(R/J_{\phi})_d \neq 0,$
which is equal to the highest exponent in the Hilbert series of $J_{\phi}.$ We see that the regulaity of $R/J_{\phi}$ is $3l+3k-3,$ or equivalently, 
the regulaity of $J_{\phi}$ is $3l+3k-2.$ 
Since the regularity is obtained at the last step of the minimal free resolution, the maximum degree 
of the second syzygies of $J_{\phi}$ is $3l+3k.$ Our goal is to show that there are no second syzygies of degree strictly smaller than 
$3l+3k.$ For that purpose, we consider the syzygies of the six syzygies of degree $2l+3k$ and the six syzygies of degree $3l+2k.$

We define the matrix 
\[\varPhi= A|B = 
\left[ \!
\begin{array}{cccccccccccc}
a_1  &a_2  &d_1  & d_2 & 0   & 0   & A_1 & A_2 & 0   & 0   & 0   & 0 \\
b_1  &b_2  &0    & 0   & g_1 & g_2 & 0   & 0   & D_1 & D_2 & 0   & 0 \\
0 & 0      &e_1  & e_2 & h_1 & h_2 & 0   & 0   & 0   & 0   & G_1 & G_2 \\
c_1 &c_2   &0    & 0   & 0   & 0   & 0   & 0   & 0   & 0   & H_1 & H_2 \\
0 & 0      &f_1  & f_2 & 0   & 0   & 0   & 0   & E_1 & E_2 & 0   & 0 \\
0 & 0      &0    & 0   & k_1 & k_2 & B_1 & B_2 & 0   & 0   & 0   & 0 \\
0 & 0      &0    & 0   & 0   & 0   & C_1 & C_2 & F_1 & F_2 & K_1 & K_2 
\end{array}\! \right]. 
\]

A syzygy of the six syzygies of degree $2l+3k$ and six syzygies of degree $3l+2k$ is a vector 
\[
 U =(v_1, v_2,\cdots, v_6, w_1,\cdots, w_6)^{t},
\]
where each component $v_i,w_i$ is a homogeneous polynomial of $x,y,z$ such that 
\begin{equation} \label{eqm}
\varPhi U = 0. 
\end{equation}
The degree of this syzygy is 
\[
 2l+3k + \deg v_1 = 3l+2k + \deg w_1.
\]
Since we want to show the non-existence of this syzygy of degree smaller than $3l+3k,$ we reduce to show that
it is impossible to have $\deg v_1 < l$ and $\deg w_1 < k.$ 

Writing the Equation \eqref{eqm} as two matrix equations as follows,
\begin{equation} \label{eqs2}
  \left\{
  \begin{array}{l}
    A v =  s \\
    B w = -s 
  \end{array} \right.
\end{equation}
Where $v = (v_1,v_2,\cdots, v_6)^t,$  $w = (w_1,w_2,\cdots, w_6)^t,$ and 
$s = (s_1,s_2,\cdots,s_6)^t.$

For example, the first equation we get by expanding $\varPhi U = 0$ is 
\[
 a_1 v_1 + a_2 v_2 + d_1 v_3 + d_2 v_4 + A_1 w_1 + A_2 w_2 = 0,
\]
which is equivalent to 
\begin{equation} \label{eqs21}
  \left\{
  \begin{array}{l}
    a_1 v_1 + a_2 v_2 + d_1 v_3 + d_2 v_4 =  s_1 \\
    A_1 w_1 + A_2 w_2= - s_1 
  \end{array} \right.
\end{equation}

It is clear that $s_1$ must be in the ideal $I_1$ generated by $a_1,a_2,d_1,d_2,$ and also in the ideal $I_2$ generated by $A_1,A_2.$ 
Therefore, we must have 
\begin{equation}
  s_1 \in I_{t}: = \langle a_1,a_2,d_1,d_2 \rangle \cap \langle A_1, A_2 \rangle. 
\end{equation}

Just prior to Equation (2.3), we have shown that 
\begin{gather*}
 \deg a_1 = \deg a_2 = \deg d_1 = \deg d_2 = l+k, \\
 \deg A_1 = \deg A_2 = 2l.
\end{gather*}
To show that $\deg v_1 < l$ is impossible, we just need to show that $\deg s_1 < 2l+k$ is impossible. It suffices to show that $(I_t)_{2l+k-1} = 0,$ 
or equivalently the Hilbert function $HF(R/I_t, 2l+k-1) = \binom{2l+k+1}{2}.$

Recall that, 
\begin{gather*}
 \langle a_1, a_2 \rangle = \langle y^{l+2k}, (x+y)^{2l+2k} \rangle : x^{l+2k}, \\
 \langle d_1, d_2 \rangle = \langle z^{l+2k}, (x+z)^{2l+2k} \rangle : x^{l+2k}, \\
 I_2= \langle A_1, A_2 \rangle = \langle (y+z)^{2l+2k}, (x+y+z)^{3l} \rangle : x^{l+2k}.
\end{gather*}

Therefore, we also have
\begin{gather*}
 I_1 = \langle y^{l+2k}, z^{l+2k}, (x+y)^{2l+2k}, (x+z)^{2l+2k}  \rangle : x^{l+2k}, \\
 I_t = I_1 \cap I_2, \\
 I_1 + I_2 = \langle y^{l+2k}, z^{l+2k}, (x+y)^{2l+2k}, (x+z)^{2l+2k}, (y+z)^{2l+2k}, (x+y+z)^{3l} \rangle : x^{l+2k}.
\end{gather*}

To compute $HF(R/{I_t}, 2l+k-1),$ we use the following exact sequence 
\[
 0 \rightarrow R/I_t \rightarrow R/I_1 \oplus R/I_2 \rightarrow R/(I_1 + I_2) \rightarrow  0.
\]

Therefore, we have
\begin{equation} \label{meq1}
\begin{split}
 HF(R/I_t, 2l+k-1) =& HF(R/I_1, 2l+k-1)+ HF(R/I_2, 2l+k-1) \\
                    & - HF(R/(I_1+I_2), 2l+k-1).
\end{split}
\end{equation}

\begin{rmk}
 In the above, we have assumed that $s_1 \neq 0.$ This is not really a restriction, since we have just considered the element $s_1$ from the first equation of $\varPhi U = 0,$
 we could equally consider the element $s_2$ from the second equation and $s_3$ from the third equation. They only differ by exchange of $x,y,z,$ and cannot be zero simultaneously.
 Otherwise, looking at the equation $B w = 0,$ the vanishing $s_1 = s_2 =s_3 =0$ would imply that 
\begin{align*}
(w_1,w_2) &= (-A_2,A_1),\\
(w_3,w_4) &= (-D_2,D_1),\\
(w_5,w_6) &= (-G_2,G_1).
\end{align*}
Then the last equation of $Bw = 0$ becomes
\begin{equation}
 \begin{split}
  (-C_1A_2 + C_2A_1) + (-F_1D_2 + F_2D_1) +(-K_1G_2 + K_2G_1) \\
  = \beta_0 (y+z)^{2l+2k} + \beta_1 (x+z)^{2l+2k} + \beta_2 (x+y)^{2l+2k} \neq 0,
 \end{split}
\end{equation}
a contradiction, since the $\beta_0, \beta_1, \beta_2$ are nonzero constants. 
\end{rmk}

Since the ideal $I_2,$ generated by two elements $A_1,A_2$ of degree $2l,$ is a complete intersection and thus has the following minimal free resolution
\[
 0 \rightarrow R(-4l) \rightarrow R(-2l)^2 \rightarrow R \rightarrow R/I_2 \rightarrow  0.
\]
Taking the $(2l+k-1)$-th graded piece of each component yields
\begin{equation} \label{eqi2}
\begin{split}
 HF(R/I_2,2l+k-1) &= HF(R,2l+k-1) + HF(R(-4l),2l+k-1) \\
                  & - 2 HF(R(-2l),2l+k-1) \\
                  & = \begin{cases}  \binom{2l+k+1}{2} - 2 \binom{k+1}{2}                     & \mbox{if } 2l \geq k, \\
                                     \binom{2l+k+1}{2} - 2 \binom{k+1}{2}+ \binom{k+1-2l}{2}  & \mbox{if } 2l < k.
\end{cases}
\end{split}
\end{equation}

The computation of the Hilbert function of the ideals $I_1, I_1+I_2$ is more complicated, which we tackle in the next subsection.

\subsection{Hilbert function computation}
To compute the Hilbert function of the ideals $I_1, I_1+I_2,$ we use the following exact sequence for an ideal $I \subset R $ and $f \in R,$
\begin{equation} \label{link}
 0 \rightarrow R(-d)/\langle I: f \rangle \xrightarrow{ \cdot f} R/I \rightarrow R/ \langle I,f \rangle \rightarrow  0,
\end{equation}
where $d = \deg f.$ Therefore, 
\begin{equation} \label{eqlk}
 HF(R/\langle I: f \rangle,t) = HF(R/I,t + d) - HF(R/\langle I,f \rangle,t + d).
\end{equation}

We introduce the following three ideals 
\begin{gather*}
L_1 = \langle y^{l+2k}, z^{l+2k}, (x+y)^{2l+2k}, (x+z)^{2l+2k}  \rangle, \\
L_2 = \langle x^{l+2k}, y^{l+2k}, z^{l+2k}, (x+y)^{2l+2k}, (x+z)^{2l+2k}  \rangle, \\
L_3 = \langle y^{l+2k}, z^{l+2k}, (x+y)^{2l+2k}, (x+z)^{2l+2k}, (y+z)^{2l+2k}, (x+y+z)^{3l}  \rangle. 
\end{gather*}

Therefore, 
\begin{gather*}
I_1 = L_1: x^{l+2k},  \\
I_1+I_2 = L_3: x^{l+2k}. 
\end{gather*}

Applying Equation \eqref{eqlk} to the cases $I = L_1, L_3,$ and $f = x^{l+2k},$ we get 
\begin{gather} \label{hf}
 HF(R/I_1,2l+k-1) = HF(R/L_1,3l+3k-1) - HF(R/L_2,3l+3k-1), \\
 HF(R/I_1+I_2,2l+k-1) = HF(R/L_3,3l+3k-1) - HF(R/J_{\phi},3l+3k-1).
\end{gather}

To compute the Hilbert function of $L_1, L_2, L_3,$ we use the results of Emsalem-Iarrobino \cite{EI95},
to translate to a question about fatpoints on $\mathbb{P}^2,$ and the work of Harbourne \cite{H96} 
which shows how to determine the dimension of a linear system on a blowup of $\mathbb{P}^2$ 
at eight or fewer points. 

\subsubsection{Inverse system}
In \cite{EI95}, Emsalem and Iarrobino proved that there is a close connection between ideals generated by powers of 
linear forms in $n$ variables and ideals of fatpoints on $\mathbb{P}^{n-1}.$ We use their results in the special case $n=3.$

Let ${p_1, \cdots, p_n} \in \mathbb{P}^{2}$ be a set of distinct points, 
\begin{align*}
p_i  &= [p_{i1} : p_{i2}: p_{i3}], \\ 
I(p_i)& = \wp_i \subseteq R'=\mathbb{K}[x',y',z']. 
\end{align*}
A fatpoint ideal is an ideal of the form
\begin{equation}
 F = \bigcap_{i=1}^{n} \wp_i^{\alpha_{i}+1} \subset R'.
\end{equation}

We also define the linear forms correspoinding to the points by
\begin{equation}
L_{p_i} = p_{i1}x + p_{i2} y + p_{i3}z \in R, \mbox{  for } 1 \leq i \leq n.
\end{equation}

Define an action of $R'$ on $R$ by partial differentiation: 
\begin{equation}
p(x',y',z') \cdot q(x,y,z) = p(\partial/\partial{x},\partial/\partial{y},\partial/\partial{z}) q(x,y,z).
\end{equation}

Since $F$ is a submodule of $R'$, it acts on $R$. The set of elements annihilated by the action of $F$ is denoted by $F^{-1}$.

\begin{thm}[Emsalem and Iarrobino \cite{EI95}]
 Let $F$ be an ideal of fatpoints 
\[
 F = \bigcap_{i=1}^{n} \wp_i^{\alpha_{i}+1},
\]
then 
\begin{equation}
 (F^{-1})_j = 
 \begin{cases}
  R_j                                     &\text{for $j\leq \max{\{\alpha_i\}}$}, \\
  L_{p_1}^{j-\alpha_1}R_{\alpha_1} + \cdots, L_{p_n}^{j-\alpha_n} R_{\alpha_n}     &\text{for $j\geq \max \{\alpha_i +1 \}$}. 
 \end{cases}
\end{equation}
and \[
     \dim_{\mathbb{K}} (F^{-1})_j = \dim_{\mathbb{K}} (R/F)_j.
    \]
\end{thm}

Suppose we have an ideal generated by powers of linear forms, and for each 
$j \in \mathbb{N},$ we wish to compute the dimension of
\[
 \langle L_{p_1}^{t_1}, \cdots, L_{p_m}^{t_m} \rangle_j.
\]
Since the $t_i$ are fixed, to apply the approach above we fix a degree $j.$ Put
\begin{equation} \label{eqfj}
F(j) = {\wp_1}^{j- t_1+1} \cap \cdots \cap {\wp_m}^{j-t_m+1}.
\end{equation}
Then
\[
\dim_{\mathbb{K}} \langle L_{p_1}^{t_1}, \cdots, L_{p_m}^{t_m} \rangle_j = \dim_{\mathbb{K}} (R/F(j))_j.
\]
Therefore, 
\begin{equation}
 HF(R/\langle L_{p_1}^{t_1}, \cdots, L_{p_m}^{t_m} \rangle, j) = \dim_{\mathbb{K}} F(j)_j.
\end{equation}

\subsubsection{Fatpoints, Divisors, and Algorithm}
There is a correspondence between the graded pieces of an ideal of fatpoints $F$ and the global sections of a certain line bundle on the
surface $X$ which is the blowup of $\mathbb{P}^2$ at the points. Let $E_i$ be the class of the
exceptional curve over the point $P_i,$ and $E_0$ the pullback of a line on $\mathbb{P}^2.$ 
The canonical divisor on $X$ is:
\[
K_X = -3E_0 + \sum_{i=1}^m E_i.
\]
We also define 
\[A_m = (m-2) E_0 - K_X. \] 
The fatpoint ideal $F(j)$ corresponds to the divisor 
\[
D_j = jE_0 -\sum_{i=1}^m a_i E_i 
\]
where $a_i = j-t_i +1,$ for $i=1,2,\cdots, m.$ Then we have 
\[
F(j)_j = H^0(D_j).
\]

To describe the algorithm to compute $h^0(D)$ for a divisor $D$ on $X,$ we need a few more definitions. A prime divisor is the class of a reduced irreducible curve on $X$ and 
an effective divisor is a nonnegative integer combination of prime divisors. We denote
the set of effective divisors by $\mathit{EFF(X)}$. A divisor whose intersection product with every effective
divisor is $\geq 0$ is called $\mathit{numerically}$ $\mathit{effective}$(nef). We define $\mathit{Neg(X)}$ as the classes
of prime divisors $C$ with $C^2 < 0$. In \cite{GHM} Proposition 3.1 and 4.1, $\mathit{Neg(X)}$ is explicitly determined, which
is the main point for the following algorithm of Geramita, Harbourne, and Migliore to compute $h^0(D)$ for any 
divisor $D$ on $X.$ 

$\boldsymbol{Algorithm}:$ \label{algm} \\
Start with $H = D,$ $N=0$.\\
If $H.C < 0 $ for some $C \in \mathit{Neg(X)},$ replace $H$ by $H-C$ and replace $N$ by $N+C$.
Eventually either $H. A_m < 0$ or $H. C \geq 0$ for all $C \in Neg(X)$.\\
In the first case, $D$ is not effective, and $h^0(D) = 0.$ \\
In the latter case, $H$ is nef and effective and we have a Zariski decomposition 
\[D = H + N,\] 
with
\[ h^0(D) = h^0(H) = (H^2 - H.K_X)/2 + 1.\]

\subsubsection{Computation} 
On $\mathbb{P}^2$, we consider the following seven points
\begin{align*}
P_1 &= [1: 0: 0],\\
P_2 &= [0: 1: 0],\\
P_3 &= [0: 0: 1],\\
P_4 &= [1: 1: 0],\\
P_5 &= [1: 0: 1],\\
P_6 &= [0: 1: 1],\\
P_7 &= [1: 1: 1].
\end{align*}

By the correspondence above, the ideal $L_1 =\langle y^{l+2k}, z^{l+2k}, (x+y)^{2l+2k}, (x+z)^{2l+2k}  \rangle$ 
corresponds to the fatpoint ideal supported at the four points $\{P_2, P_3, P_4, P_5\}.$ To compute $HF(R/L_1, 3l+3k-1),$  we consider the 
corresponding fatpoint ideal 
\[
  F_1 = {\wp_2}^{2l + k} \cap {\wp_3}^{2l + k } \cap {\wp_4}^{l + k} \cap {\wp_5}^{l+k},
\]
and the divisor on the surface $X$ which is the blowup of $\mathbb{P}^2$ at the four points,
\[
  D_1 =  (3l+3k-1) E_0 - (2l + k)(E_2+E_3) - (l+k)(E_4+E_5).
\]

\begin{prop}\label{pl1a}
\begin{equation} 
HF(R/L_1, 3l+3k-1)= h^0(D_1) = 
\begin{cases}   3k^2+2lk-l-k                   & \mbox{if } l \geq k-1, \\
                \frac{-l^2+6lk+5k^2-3l-k}{2}   & \mbox{if } l \leq k-1.
\end{cases}
\end{equation}
\end{prop}

\begin{proof}
On the surface $X,$ we have the following negative classes
\begin{align*}
 C_1 &= E_0 - E_2 - E_3, \\
 C_2 &= E_0 - E_2 - E_4, \\
 C_3 &= E_0 - E_2 - E_5, \\
 C_4 &= E_0 - E_3 - E_4, \\
 C_5 &= E_0 - E_3 - E_5, \\
 C_6 &= E_0 - E_4 - E_5.
\end{align*}

Since
\begin{equation*}
 D_1 \cdot C_1 = (3l+3k-1) - 2(2l+k) = -l+k-1,
\end{equation*}
we consider two cases: $l \geq k-1$ and $l<k-1.$  

In the first case, 
\begin{equation*}
 \begin{split}
 D_1 \sim D_1' &= D_1 - (l-k+1)C_1 \\
               &= (2l+4k-2)E_0 - (l+2k-1)(E_2+E_3) - (l+k)(E_4+E_5), 
 \end{split}
\end{equation*}
which meets $C_1$ nonnegatively. 

Moreover, 
\[
 D_1' \cdot C_i= k-1 \geq 0, \text{ for } i = 2,3,4,5;
\]
and
\[
 D_1' \cdot C_6= 2k-2 \geq 0.
\] 

We conclude that $D_1'$ is nef. Since
\begin{align*}
 D_1'^2         & = 2(2l+3k-1)(k-1),\\
 D_1' \cdot K_X & = -2l-6k+4,
\end{align*}
we have 
\begin{equation}
 h^0(D_1) = h^0(D_1') = 3k^2+2lk-l-k.
\end{equation}

In the other case $l < k-1,$ it is easy to check that
\[
 D_1 \cdot C_i \geq 0, \text{ for } i =1,2,3,4,5,6.
\]
Therefore $D_1$ is nef. Since
\begin{align*}
 D_1^2         &= -l^2+6lk+5k^2-6l-6k+1,\\
 D_1 \cdot K_X &= -3l-5k+3,
\end{align*}
we have 
\begin{equation}
 h^0(D_1) = \frac{-l^2+6lk+5k^2-3l-k}{2}.
\end{equation}
\end{proof}

Similarly, the ideal $L_2$ corresponds to the fatpoint ideal supported at the points $\{P_1, P_2, P_3, P_4, P_5\}.$ To compute $HF(R/L_2,3l+3k-1),$ we consider 
\begin{gather}
  F_2 = {\wp_1}^{2l + k} \cap {\wp_2}^{2l + k} \cap {\wp_3}^{2l + k }\cap {\wp_4}^{l + k } \cap {\wp_5}^{l+k},
\end{gather}
and the divisor on the surface $X$ which is the blowup of $\mathbb{P}^2$ at the five points, 
\begin{equation*}
  D_2 =  (3l+3k-1) E_0 - (2l + k)(E_1+E_2+E_3) - (l+k)(E_4+E_5).
\end{equation*}

\begin{prop} \label{pl2}
\begin{equation}
HF(R/L_2, 3l+3k-1)= h^0(D_2)=
\begin{cases}   \frac{5k^2-3k}{2}.             & \mbox{if } l \geq k-1, \\
                \frac{-l^2+2lk+4k^2-l-2k}{2}   & \mbox{if } l < k-1.
\end{cases}
\end{equation}
\end{prop}

\begin{proof}
The proof is very similar to that of Proposition \ref{pl1a} and the key is to detemine the negative classes on $X.$
On the surface $X,$ we have the following negative classes
\begin{align*}
 C_1 & = E_0 - E_1 - E_3 - E_5, \\
 C_2 & = E_0 - E_1 - E_2 - E_4, \\
 C_3 & = E_0 - E_2 - E_3,       \\
 C_4 & = E_0 - E_2 - E_5, \\
 C_5 & = E_0 - E_3 - E_4, \\
 C_6 & = E_0 - E_4 - E_5.
\end{align*}

Since $D_2 \cdot C_1 = -2l-1<0,$ we subtract $(l+1)C_1$ from $D_2$ to get 
\[
 D_2' = (2l+3k-2) E_0 - (l + k-1)(E_1+E_3)- (2l+k)E_2 - (l+k)E_4 -(k-1)E_5,
\]
then $D_2' \cdot C_1 >0,$ but
\[
 D_2'\cdot C_2 = -2l-1<0,
\]
we subtract $(l+1)C_2$ from $D_2'$ to get
\[
 D_2'' = (l+3k-3) E_0 - (k-2)E_1- (l+k-1)(E_2+E_3) - (k-1)(E_4+E_5).
\]
Since $D_2'' \cdot C_3  = -l+k-1,$ we consider two cases: $l \geq k-1$ and $l< k-1.$  

\smallskip

Case 1: $l \geq k-1.$ We subtract $(l-k+1)C_3$ from $D_2''$ to get 
\[
 D_2'''= (4k-4)E_0 - (k-2)E_1- (2k-2)(E_2+E_3) - (k-1)(E_4+E_5).
\]
It is easy to check $D_2'''$ is nef. Since
\begin{align*}
 D_2'''^2         &= 5k^2-8k+2,\\
 D_2''' \cdot K_X &= -5k+4,
\end{align*}
we have 
\begin{equation}
 h^0(D_2) = \frac{5k^2-3k}{2}.
\end{equation}

\smallskip

Case 2: $l < k-1.$ In this case, $D_2''$ is nef. Since
\begin{align*}
 D_2''^2         &= 4k^2+2lk-l^2-2l-6k+1,\\
 D_2'' \cdot K_X &= -l-4k+3,
\end{align*}
we have 
\begin{equation}
 h^0(D_2) = \frac{-l^2+2lk+4k^2-l-2k}{2}.
\end{equation}
\end{proof}

Similarly, the ideal $L_3$ corresponds to the fatpoint ideal supported at the points $\{P_2, P_3, P_4, P_5, P_6, P_7\}.$ To compute $HF(R/L_3,3l+3k-1),$ we consider 
\begin{gather}
  F_3 = {\wp_2}^{2l + k} \cap {\wp_3}^{2l + k }\cap {\wp_4}^{l + k } \cap {\wp_5}^{l+k} \cap {\wp_6}^{l+k} \cap {\wp_7}^{3k},
\end{gather}
and the divisor on the surface $X$ which is the blowup of $\mathbb{P}^2$ at the six points
\begin{equation*}
  D_3 =  (3l+3k-1) E_0 - (2l + k)(E_2+E_3) - (l+k)(E_4+E_5+E_6) - 3kE_7.
\end{equation*}

\begin{prop} \label{pl3}
\begin{equation}
HF(R/L_3, 3l+3k-1)= h^0(D_3)= 
\begin{cases} \frac{-k^2-k}{2}+ l(2k-1)   &\mbox{if } l \geq \frac{k+1}{2}, \\
2l^2-2l         & \mbox{if } l \leq \frac{k}{2}. \end{cases}  
\end{equation}
\end{prop}

\begin{proof}
 On the surface $X,$ we have the following negative classes,
\begin{align*}
 C_1 & = E_0 - E_2 - E_5 - E_7, \\
 C_2 & = E_0 - E_3 - E_4 - E_7, \\
 C_3 & = E_0 - E_2 - E_3 - E_6, \\
 C_4 & = E_0 - E_4 - E_5,       \\
 C_5 & = E_0 - E_4 - E_6, \\
 C_6 & = E_0 - E_5 - E_6, \\
 C_7 & = E_0 - E_2 - E_4, \\
 C_8 & = E_0 - E_3 - E_5, \\
 C_9 & = E_0 - E_6 - E_7.
\end{align*}

Since the computation is very similar, we just sketch the computation.
If $l \geq 2$ and $k \geq 2,$ after subtracting multiples of $C_1,C_2,C_3$ from $D_3,$ we get the divisor
\[
 D_3' = (2l+k-4) E_0 -(l-2)(E_2+E_3)-(l-1)(E_4+E_5)-(k-1)E_6-(k-2)E_7.
\]
We have 
\begin{align*}
 D_3' \cdot C_4 &= k-2,\\
 D_3' \cdot C_5 &= D_3' \cdot C_6 =l-2, \\
 D_3' \cdot C_7 &= D_3' \cdot C_8 = k-1. 
\end{align*}

Since
\[
 D_3' \cdot C_9 = 2l-k-1,
\]
we consider two cases: 

\smallskip

Case 1: $2l\geq k+1.$ In this case, we have
\begin{align*}
 D_3'^2         &= -k^2+4kl-4l-2k+1,\\
 D_3' \cdot K_X &= -2l-k+3.
\end{align*}
The formula for $h^0(D_3)$ follows in this case. 

\smallskip

Case 2: $2l \leq k.$ We subtract $(k+1-2l)C_9$ from $D_3'$ to get 
 $$ D_3'' =(4l-5) E_0 -(l-2)(E_2+E_3) -(l-1)(E_4+E_5)-(2l-2)E_6-(2l-3)E_7. $$

It is easy to check that $D_3''$ is nef. We have 
\begin{align*}
 D_3''^2 &= 4l^2-8l+2, \\
 D_3'' \cdot K_X &= -4l+4,
\end{align*}
then
\begin{equation}
 h^0(D_3'') = 2l^2-2l.
\end{equation}

If $l=1,$ then $D_3$ is not effective and $h^0(D_3) = 0.$

If $k=1,$ and $l \geq 2,$ then   
\[
 D_3 \sim D_3''= (2l-2) E_0 -(l-2)(E_2+E_5) - l(E_3+E_4) - E_6
\]
We have
\[
  h^0(D_3)=  h^0(D_3'') = l-1.
\]
\end{proof}



%

\subsection{Conclusion of the proof}
Recall that
\begin{gather} \label{hf}
 HF(R/I_1,2l+k-1) = HF(R/L_1,3l+3k-1) - HF(R/L_2,3l+3k-1), \\
 HF(R/I_1+I_2,2l+k-1) = HF(R/L_3,3l+3k-1) - HF(R/J_{\phi},3l+3k-1).
\end{gather}

Combining the forumlas in Propositions \ref{pl1a},\ref{pl2},\ref{pl3}, we have 
 \[
  HF(R/I_1,2l+k-1) = \frac{k^2+k}{2} +l(2k-1).
 \]

Since $HF(R/J_{\phi},3l+3k-1) = 0,$ we have
 \[
  HF(R/(I_1+I_2),2l+k-1) =  \frac{-k^2-k}{2}+ l(2k-1).             
 \]

By Equations \eqref{meq1}, \eqref{eqi2}, we get the formula 
\[
 HF(R/I_t, 2l+k-1) =  \binom{2l+k+1}{2},
\]
which implies that $(I_t)_{2l+k-1} = 0,$ exactly what we want 
to conclude that there are no first syzygies of degree $> \max(2l+3k,3l+2k)$
and our constructed syzygies generate all the first syzygies of the ideal $J_{\phi}.$

Since the Hilbert series of $R/J_{\phi}$ is given by 
\[ 
\frac{1-3t^{l+2k} - 3t^{2l+2k} - t^{3l} + 6 t^{2l+3k} + 6t^{3l+2k} - 6t^{3l+3k}}{(1-t)^3},
\]
we know that there are exactly six second syzygies of degree $3l+3k.$
\begin{thm0} \label{thm3v}
The minimal free resolution of the ideal 
$$J_{\phi} = \langle x^{l+2k}, y^{l+2k}, z^{l+2k}, (x+y)^{2l+2k}, (x+z)^{2l+2k}, (y+z)^{2l+2k}, (x+y+z)^{3l} \rangle$$ 
is given by 
$$0 \xrightarrow{} R(-3l-3k)^6 \xrightarrow{d_3}
\begin{array}{c}
 R(-2l-3k)^6\\
\oplus\\
R(-3l-2k)^6 \\
\end{array}
\xrightarrow{}
\begin{array}{c}
R(-l-2k)^3\\ 
\oplus\\
R(-2l-2k)^3 \\
 \oplus \\
R(-3l)\\
\end{array}
\xrightarrow{}
R \xrightarrow{}
R/J_{\phi} \xrightarrow{} 0.$$
\end{thm0}

\begin{rmk}
Even though we have shown there must be six second syzygies of degree $3l+3k,$ we have not been able to find an explicit set of six generators. 
It would be very desirable to find such an explicit set, thus giving the last differential of the minimal free resolution.  
\end{rmk}

\begin{rmk}
For the Postnikov-Shapiro conjecture for $n \geq 4,$ The higher syzygies can probably be constructed inductively from the subideals. But the main difficulty is to show the exactness 
of the complex constructed. 
\end{rmk}

\noindent{\bf Acknowledgments} Thanks to Hal for his support, encouragement and 
many useful conversations. Thanks to Bernd Sturmfels, Madhusudan Manjunath, Frank Schreyer, Alexandra Seceleanu,
Michael DiPasquale, Brian Benson and Jerzy Weyman for helpful conversations. 

\bibliographystyle{amsplain}

\begin{thebibliography}{10}

\bibitem{B00} {\scshape D.\ Buchsbaum},
       On commutative algebra and characteristic-free representation theory.
       {\em Journal of Pure and Applied Algebra}, {\bf 152} (2000), 41-48.

\bibitem{BE74} {\scshape D.\ Buchsbaum,D.\ Eisenbud},
     Some structure theorems for finite free resolutions.
    {\em Advances in Math.},{\bf 12} (1974), 84-139.

\bibitem{BDP10} {\scshape B.\ Benson, C.\ Deeparnab, T.\ Prasa},
          $G$-parking functions, acyclic orientations and spanning trees.
          {\em Discrete Math.} {\bf 310} (2010), no.8, 1340-1353.

\bibitem{BE73} {\scshape D.\ Buchsbaum, D.\ Eisenbud},
          What makes a complex exact?
         {\em J. Algebra}, {\bf 25} (1973), 259-268.

\bibitem{EI95} {\scshape J.\ Emsalem, A. \ Iarrobino},
         Inverse system of a symbolic power I,
         {\em J. Algebra}, {\bf 174} (1995), 1080-1090.

\bibitem{E95} {\scshape D.\ Eisenbud}, 
       \textit{Commutative Algebra with a view towards Algebraic Geometry}, 
      Springer-Verlag, Berlin-Heidelberg-New York, 1995.

\bibitem{EI00} {\scshape D.\ Eisenbud}, 
       \textit{Geometry of Syzygies}, Springer-Verlag.

\bibitem{EN73} {\scshape J.A.\ Eagon, \scshape D.\ Northcott},
      On the Buchsbaum-Eisenbud theory of finite free resolutions.
      {\em J. Reine Angew. Math.}, 262/263 (1973), 205-219. 

\bibitem{G93} {\scshape A.\ Gabrielov},
          Abelian avalanches and Tutte polynomials. 
        {\em Phys. A} 195 (1993), no. 1-2, 253-274. 

\bibitem{G96} {\scshape A.\ Geramita},
        Inverse Systems of Fat Points: Waring's Problem, Secant Varieties
        of Veronese Varieties, and Parameter Spaces for Gorenstein Ideals,
        {\em Queens Papers in Pure and Applied Mathematics} {\bf 102} (1996),
        1-114.

\bibitem{GHM} {\scshape A.V.\ Geramita, B.\ Harbourne, J.\ Migliore},
Classifying Hilbert functions of fat point subschemes in $\mathbb{P}^2$.
\textit{Collect. Math.} {\bf 60} (2009), no. 2, 159-192.

\bibitem{GS98} {\scshape A.\ Geramita, H.\ Schenck }, 
            Fatpoints, inverse systems, and piecewise polynomial functions,
            {\em J. Algebra}, {\bf 204} (1998), 116--128.

\bibitem{H96}{\scshape B.\ Harbourne}, 
         Rational surfaces with $K^2 > 0$.
         {\em Proc. Amer. Math. Soc.}, {\bf 124} (1996), 727-733. 

\bibitem{H97}{\scshape B.\ Harbourne},  
         Anticanonical rational surfaces. 
         {\em Trans. Amer. Math. Soc.}, {\bf 349} (1997), 1191-1208.

\bibitem{I97}  {\scshape A.\ Iarrobino},
         Inverse system of a symbolic power. III. Thin algebras and fat points.
         \em{Compositio Math.} 108 (1997), no.3, 319-356.

\bibitem{MM031} {\scshape J.\ Migliore, \scshape R.\ M Mir\'{o}-Roig},
         On the minimal free resolution of $n+1$ general forms.
         \em{Trans. Amer. Math. Soc.} 355 (2003), no. 1, 1-36. 

\bibitem{MM032} {\scshape J.\ Migliore, \scshape R.\ M Mir\'{o}-Roig},
         Ideals of general forms and the ubiquity of the weak Lefschetz property. 
        \em{J. Pure Appl. Algebra } 182 (2003), no. 1, 79-107. 

\bibitem{MSW12}  {\scshape M.\ Manjunath, Frank-Olaf Schreyer, John Wilmes},
          Minimal Free Resolutions of the $G$-parking Function Ideal and the Toppling Ideal.
         \em{arXiv:1210.7569}

\bibitem{MF12} {\scshape F.\ Mohammadi, F.\ Shokrieh},
          Divisors on graphs, Connected flags, and Syzygies. 
         \em{arXiv:1210.6622}

\bibitem{PS02}{\scshape A.\ Postnikov, B.\ Shapiro},
         Monotone monomial ideals and their equisupported deformations.
         preprint, March 2002.

\bibitem{PS04}{\scshape A.\ Postnikov, \scshape B.\ Shapiro},
         Trees, parking functions, syzygies, and deformations of monomial ideals,
         {\em Trans. Amer. Math. Soc.}, {\bf 356} (2004), 3109-3142.

\bibitem{S04}{\scshape H.\ Schenck},
         Linear systems on a special rational surface,
         {\em Math. Res. Lett. } {\bf 11} (2004), 697-713.
\end{thebibliography}

\end{document}